\documentclass[12pt]{amsart}
\usepackage{amstext,amsfonts,amssymb,amscd,amsbsy,amsmath,verbatim,fullpage,mathrsfs}
\usepackage[alphabetic,lite,backrefs]{amsrefs} 
\usepackage[bookmarks, colorlinks=true, linkcolor=blue, citecolor=blue, urlcolor=blue, pagebackref=true]{hyperref}
\usepackage{ifthen}
\usepackage{color,tikz}
\usepackage{amsthm}
\usepackage{latexsym}
\usepackage[all]{xy}
\usepackage{enumitem}
\usepackage{tikz-cd}									    					
\usepackage{extarrows}
\usepackage{mathtools}

\usepackage{float}	
\usepackage{parskip}

\newtheorem{lemma}{Lemma}[section]

\newtheorem{prop}[lemma]{Proposition}
\newtheorem{cor}[lemma]{Corollary}

\newtheorem{claim*}{Claim}
\newtheorem{thm}[lemma]{Theorem}
\newtheorem*{thmA}{Theorem A}
\newtheorem*{thmB}{Theorem B}
\newtheorem*{thmC}{Theorem C}
\newtheorem*{thmD}{Theorem D}
\newtheorem{defn}[lemma]{Definition} 
\newtheorem{notation}[lemma]{Notation} 

\newtheorem{question}[lemma]{Question}

\theoremstyle{remark}
\newtheorem{remark}[lemma]{Remark}
\newtheorem{example}[lemma]{Example}

\newcommand{\Spec}{\operatorname{Spec}}

\newcommand{\im}{\operatorname{im}}

\newcommand{\Proj}{\operatorname{Proj}}

\newcommand{\Pic}{\operatorname{Pic}}

\newcommand{\Gr}{\operatorname{Gr}}
\newcommand{\cO}{{\mathcal O}}
\newcommand{\cP}{{\mathrm{Par}}}
\newcommand{\cM}{{\mathrm{Med}}}
\newcommand{\cL}{{\mathrm{Low}}}
\newcommand{\cH}{{\mathrm{High}}}

\newcommand{\ff}{\mathbf f}

\newcommand{\kk}{\mathbf k}

\newcommand{\codim}{\operatorname{codim}}

\newcommand{\FF}{\mathbb{F}}
\newcommand{\Prob}{\operatorname{Prob}}
\newcommand{\Density}{\operatorname{Density}}

\newcommand{\defi}[1]{\textsf{#1}} 
\newcommand{\Par}{\mathscr D}
\newcommand{\bad}{\text{bad}}
\newcommand{\good}{\text{good}}
\newcommand{\deghat}{\widehat{\deg}}

\newcommand{\PP}{\mathbb P}
\renewcommand{\AA}{\mathbb A}
\newcommand{\QQ}{\mathbb Q}

\newcommand{\ZZ}{\mathbb Z}
\newcommand{\VV}{\mathbb V}

\title{A probabilistic approach to systems of parameters and Noether normalization}
\author{Juliette Bruce}
\address{Department of Mathematics, University of Wisconsin, Madison, WI}
\email{\href{mailto:juliette.bruce@math.wisc.edu}{juliette.bruce@math.wisc.edu}}
\urladdr{\url{http://math.wisc.edu/~juliettebruce/}}
\author{Daniel Erman}
\email{\href{mailto:derman@math.wisc.edu}{derman@math.wisc.edu}}
\urladdr{\url{http://math.wisc.edu/~derman/}}

\begin{document} 

\setlength{\abovedisplayskip}{5.25pt}
\setlength{\belowdisplayskip}{5.25pt}

\thanks{The first author was partially supported by the NSF GRFP under Grant No. DGE-1256259. The second author was partially supported by NSF grants DMS-1302057 and DMS-1601619.}

\maketitle
\begin{abstract}
We study systems of parameters over finite fields from a probabilistic perspective, and use this to give the first effective Noether normalization result over a finite field.  Our central technique is an adaptation of Poonen's closed point sieve, where we sieve over higher dimensional subvarieties, and we express the desired probabilities via a zeta function-like power series that enumerates higher dimensional varieties instead of closed points.
This also yields a new proof of a recent result of Gabber-Liu-Lorenzini and Chinburg-Moret-Bailly-Pappas-Taylor on Noether normalizations of projective families over the integers.
\end{abstract}
\setcounter{section}{1}

Given an $n$-dimensional projective scheme $X\subseteq \PP^r$ over a field, Noether normalization  says that we can find homogeneous polynomials that induce a finite morphism $X\to \PP^n$.  Such a morphism is determined by a system of parameters, namely by choosing homogeneous polynomials $f_0,f_1,\dots,f_n$ of degree $d$ where $X\cap V(f_0, f_1, \dots,f_n)=\emptyset$.  While over an infinite field any generic choice of linear polynomials will work, over a finite field we can ask:
\begin{question}\label{quest:main}Let $\FF_q$ be a finite field and $X\subseteq \PP^r_{\FF_q}$ be an $n$-dimensional closed subscheme.
\begin{enumerate}[noitemsep]
\item What is the probability that a random choice of polynomials of degree $d$ will yield a finite morphism $X\to \PP^n$?
\item Can one effectively bound the degrees $d$ for which such a finite morphism exists?
\end{enumerate}
\end{question}
We answer these questions by studying the distribution of systems of parameters from both a geometric and probabilistic viewpoint.  It is also useful to analyze partial systems of parameters, so for $k\leq n$ we say that $f_0,f_1,\dots ,f_{k}$ \defi{ are parameters on $X$} if
\[
\dim \VV(f_0,f_1,\dots,f_{k})\cap X = \dim X - (k+1).
\]
By convention, the empty set has dimension $-1$.  

For the geometric side, we fix a field $\kk$ and let $S=\kk[x_0,\dots,x_r]$ be the coordinate ring of $\PP^r_\kk$.  We write $S_d$ for the vector space of degree $d$ polynomials in $S$. In \S\ref{sec:geometric}, we define a scheme $\Par_{k,d}(X)$ parametrizing collections that do not form parameters. The $\kk$-points of $\Par_{k,d}(X)$ are
\[
\Par_{k,d}(X)(\kk)=\{(f_0,f_1,\dots,f_k) \text{ that are \textbf{not} parameters on $X$}\} \subset \underbrace{S_d\times \dots \times S_d}_{k+1 \text{ copies }}.
\]
We bound the codimension of these closed subschemes of the affine space $S_d^{\oplus k+1}$.

\begin{thmA}
Let $X\subseteq \PP_{\kk}^r$ be an $n$-dimensional closed subscheme.  We have:
\[
\codim\Par_{k,d}(X) =
\begin{cases} \geq \binom{n-k+d}{n-k} &\mbox{if } k<n \\ 
=1& \mbox{if } k=n.
\end{cases}
\]
\end{thmA}
\noindent This generalizes several results from the literature: the case $k=n$ is a classical result about Chow forms~\cite[3.2.B]{gkz}; the case $d=1$ is a classical result about determinantal varieties~\cite{macaulay-determinantal}; and the case $k=0$ appears in~\cite[Lemme~2.3]{benoist}.  If $k<n$, then the codimension grows as $d\to \infty$ and this factors into our asymptotic analysis over finite fields.  It also leads to a uniform convergence result that allows us to go from a finite field to $\ZZ$.  

For the probabilistic side, we work over a finite field $\FF_q$ and compute the asymptotic probability that random polynomials $(f_0,f_1,\dots,f_k)$ of degree $d$ are parameters on $X$. This is inspired by~\cite{poonen}, and it forms the heart of the paper. There is a bifurcation between the maximal and submaximal cases, reflecting Theorem~A.  

\begin{thmB}\label{thm:B}
Let $X\subseteq \PP_{\FF_q}^r$ be an $n$-dimensional closed subscheme.  Then the
asymptotic probability that random polynomials $(f_0, f_1, \dots,f_k)$ of degree $d$ are parameters on $X$ is
\[
\lim_{d\to \infty} \Prob\left(\begin{matrix}(f_0,\dots,f_{k}) \text{ of degree $d$ } \\ \text{are parameters on $X$}\end{matrix}\right)=\begin{cases} 1 &\mbox{if } k<n \\ 
\zeta_X(n+1)^{-1} & \mbox{if } k=n \end{cases} 
\]
where $\zeta_X(s)$ is the arithmetic zeta function of $X$.
\end{thmB}
The maximal case $k=n$ is due to Bucur and Kedlaya~\cite[Theorem~1.2]{bucur-kedlaya}, and is proven using Poonen's closed point sieve.
For submaximal cases where $k<n$, we adapt Poonen's technique by sieving over closed subvarieties of dimension $n-k$.   Just as the case $k=n$ depends on the zeta function of $X$, which counts points in $X$ of varying degrees, we show that each case $k<n$ depends on a power series that counts $(n-k)$-dimensional subvarieties of varying degrees.  The full computation of these probabilities appears in Theorem~\ref{thm:main finite field}, while the following corollary of Theorem~\ref{thm:main finite field} computes the first error term.

\begin{cor}\label{cor:error}
Let $X\subseteq \PP^r_{\FF_q}$ be a $n$-dimensional closed subscheme and let $k<n$.  Then
\[
\lim_{d\to \infty} \frac{\Prob\left(\begin{matrix}(f_0,\dots,f_{k}) \text{ of degree $d$} \\ \text{ are \underline{not} parameters on $X$}\end{matrix}\right)} {q^{-(k+1)\binom{n-k+d}{n-k}}} = \#\left\{\begin{matrix}\text{$(n-k)$-planes } L\subseteq \PP^r_{\FF_q}\\\text{such that }  L\subseteq X\end{matrix}\right\}.
\]
\end{cor} 
When $k<n$, the probability in Theorem~B behaves like $1-O\left(q^{-(k+1)\binom{n-k+d}{n-k}}\right)$, with the precise convergence governed by the number of linear spaces of dimension $n-k$ in $X$.  It is thus more difficult to find parameters when $X$ contains lots of linear spaces, see Example~\ref{ex:linear}.

The main difficulty in Theorem~B and Theorem~\ref{thm:main finite field} is bounding the error of the sieve.  We control this error via a uniform lower bound for Hilbert functions provided in Lemma~\ref{lem:key Hilbert function}.


From these probabilistic results we are able to provide an answer to Question~\ref{quest:main}.(2). The bound is in terms of the sum of the degrees of the irreducible components.  If $X\subseteq \PP^r$ has minimal irreducible components $V_1,\dots,V_s$ (considered with the reduced scheme structure), then we define $\deghat(X):=\sum_{i=1}^s \deg(V_i)$ (see also Defintion~\ref{defn:deghat}.)

\begin{thmC}\label{thmC}
Let $X\subseteq \PP^r_{\FF_q}$ where $\dim X = n$. If $\max\left\{d,\frac{q}{d^{n}}\right\}\geq \deghat(X)$ and
\[
d > \log_q \deghat(X) + \log_q n+ n\log_q d
\]
then there exist $f_0, f_1,\dots,f_n$ of degree $d^{n+1}$ inducing a finite morphism $\pi: X\to \PP^n_{\FF_q}$.
\end{thmC}
The bound is asymptotically optimal in $q$.  Namely, if we fix $\deghat(X)$, then as $q\to \infty$, the bound becomes $d=1$.  Thus, linear Noether normalizations exist if $q\gg \deghat(X)$.  For a fixed $q$, we expect the bound could be significantly improved. This is interesting even in the case $\dim X=0$, where it is related to Kakeya type problems over finite fields~\cites{ellenberg-erman,eot}.

Theorem~C provides the first explicit bound for Noetherian normalization over a finite field.  (One could potentially derive an explicit bound from Nagata's argument in~\cite[Chapter I.14]{nagata}, though the inductive nature of that construction would at best yield a bound that is multiply exponential in the largest degree of a defining equation of $X$.)

After computing the probabilities over finite fields, we combine these analyses and characterize the distribution of parameters on projective $B$-schemes where $B=\ZZ$ or $\FF_q[t]$. We use standard notions of density for a subset of a free $B$-module; see Definition~\ref{defn:density}. 
\begin{thmD}\label{thm:D}
If $X\subseteq \PP^r_{B}$ is a closed subscheme whose general fiber over $B$ has dimension $n$, then 	\[
	\lim_{d\to \infty} \Density\left\{\begin{matrix}(f_0, f_1, \dots,f_{k}) \text{ of degree $d$} \text{ that restrict} \\ \text{ to parameters on $X_p$ for all $p$}\end{matrix}\right\} = 
	\begin{cases} 1 &\mbox{if } k<n \\ 
0& \mbox{if } k=n \mbox{ and all } d.\end{cases}
	\]
\end{thmD}
\noindent The density over $B$ thus equals the product over all the fibers of the asymptotic probabilities over $\FF_q$.  In the case $B=\ZZ$, our proof relies on Ekedahl's infinite Chinese Remainder Theorem~\cite[Theorem 1.2]{ekedahl} combined with Proposition~\ref{prop:precise estimate}, which illustrates uniform convergence in $p$ for the asymptotic probabilities in Theorem B.  In the case $B=\FF_q[t]$, we use Poonen's analogue of Ekedahl's result~\cite[Theorem 3.1]{poonen-squarefree}.

When $k=n$, an analogue of Theorem~D for smoothness is given by Poonen's~\cite[Theorem~5.13]{poonen}.  Moreover,  it is believed that there are no smooth hypersurfaces of degree $>2$ over $\ZZ$.  By contrast, the density zero subset from Theorem~D turns out to always be nonempty.
This leads to a new proof of a recent result about uniform Noether normalizations.
\begin{cor}\label{thm:noether normalization over ZZ}
Let $B=\ZZ $ or $\FF_q[t]$.  Let $X\subseteq \PP^r_B$ be a closed subscheme.  If each fiber of $X$ over $B$ has dimension $n$, then for some $d$, there exist homogeneous polynomials $f_0,f_1,\dots,f_n\in B[x_0,x_1,\dots,x_r]$ of degree $d$ inducing a
 finite morphism $\pi: X\to \PP^n_B$.
\end{cor}
Theorem~D shows that the collections defining a finite map $\pi$ have density zero, even as $d\to \infty$.  Thus the existence of $\pi$ is subtle and perhaps unexpected.  Corollary~\ref{thm:noether normalization over ZZ} is a special case of a recent result of Chinburg-Moret-Bailly-Pappas-Taylor \cite[Theorem~1.2]{cmbpt} and of Gabber-Liu-Lorenzini \cite[Theorem 8.1]{gabber-liu-lorenzini}.  Our proof of the corollary involves two steps.  We first use Theorem~D  to choose a submaximal collection $f_0, f_1, \dots,f_{n-1}$ of parameters on $X$.  This yields a scheme $X':=X\cap \VV(f_0,f_1,\dots,f_{n-1})$ with $0$-dimensional fibers over $\Spec(B)$, and we then use that $\Pic(X')$ is torsion to find the final section $f_n$. 

Corollary~\ref{thm:noether normalization over ZZ} does not require the full strength of Theorem~D, as we only need that the set of good choices for $f_0,f_1,\dots,f_{n-1}$ is nonempty, and the proofs in~\cite{cmbpt,gabber-liu-lorenzini} do use simpler techniques when selecting $f_0,f_1,\dots,f_{n-1}$.  For the final section $f_n$, they use essentially the same technique as we do.  Since we rely on the infinite Chinese Remainder Theorems of Ekedahl and Poonen, Corollary~\ref{thm:noether normalization over ZZ} only recovers the results in \cite{cmbpt,gabber-liu-lorenzini} for $\ZZ$ and $\FF_q[t]$, but not for more general number fields or function fields.

Corollary~\ref{thm:noether normalization over ZZ} can fail when $B$ is any of $\QQ[t]$ or $\ZZ[t]$ or $\FF_q[s,t]$, as in those cases, the Picard group of a finite cover of $\Spec B$ can fail to be torsion.  See \S\ref{sec:examples} for explicit examples and counterexamples and see \cite{cmbpt,gabber-liu-lorenzini} for generalizations and applications.

It would be interesting to produce an effective version of Corollary~\ref{thm:noether normalization over ZZ} similar to Theorem~C.  One would need to bound the height of $f_0,f_1,\dots,f_{n-1}$ when applying Theorem~D, and combining this with  effective bounds on the size of the class group of a number field.

There are a few earlier results related to Noether normalization over the integers.  For instance~\cite{moh} shows that Noether normalizations of semigroup rings always exist over $\ZZ$; and ~\cite[Theorem~14.4]{nagata} implies that given a family over any base, one can find a Noether normalization over an open subset of the base.  Relative Noether normalizations play a key role in~\cite[\S5]{achinger}.
 There is also the incorrect claim in~\cite[p.~124]{zariski-samuel} that Noether normalizations exist over any infinite base ring (see~\cite{abhyankar-kravitz}).  
 Brennan and Epstein~\cite{brennan-epstein}
  analyze the distribution of systems of parameters from a different perspective, introducing the notion of a generic matroid to relate various different systems of parameters. 

This paper is organized as follows. \S\ref{sec:background} gathers background results and \S\ref{sec:Hilbert function} involves a key lower bound on Hilbert functions.  \S\ref{sec:geometric} contains our geometric analysis of parameters including a proof of Theorem~A. \S\ref{sec:probabilistic} and \S\ref{sec:error} contain our probabilistic analysis of parameters over finite fields: \S\ref{sec:probabilistic} proves Theorem~B  while \S\ref{sec:error} gives the more detailed description  via an analogue of the zeta function enumerating $(n-k)$-dimensional subvarieties.  \S\ref{sec:ZZ} contains our analysis over $\ZZ$ including proofs of Theorems~C and~\ref{thm:noether normalization over ZZ} and related corollaries.  \S\ref{sec:examples} contains examples.

\section*{Acknowledgements}  We thank Nathan Clement, David Eisenbud, Jordan S.\ Ellenberg, Mois\'es Herrad\'on Cueto, Craig Huneke, Kiran Kedlaya, Brian Lehmann, Dino Lorenzini, Bjorn Poonen, Anurag Singh, and Melanie Matchett Wood for their helpful conversations and comments.  The computer algebra system \texttt{Macaulay2} \cite{M2} provided valuable assistance throughout our work.  

\section{Background}\label{sec:background}
In this section, we gather some algebraic and geometric facts that we will cite throughout.

\begin{lemma}\label{lem:nakayama}
Let $\kk$ be a field and let $R$ be a $(k+1)$-dimensional graded $\kk$-algebra where $R_0=\kk$. If $f_0, f_1, \dots, f_k$ are homogeneous elements of degree $d$ and $R/(f_0, f_1, \dots,f_k)$ has finite length, then the extension $\kk[z_0, z_1, \dots,z_k]\to R$ given by $z_i\mapsto f_i$ is a finite extension.
\end{lemma}
\begin{proof}
See~\cite[Theorem~1.5.17]{bruns-herzog}.
\end{proof}

\begin{defn}\label{defn:deghat}
Let $X\subseteq \PP^r$ be a projective scheme with minimal irreducible components $V_1,\dots,V_s$ (considered with the reduced scheme structure). We define $\deghat(X):=\sum_{i=1}^s \deg(V_i)$.  For a subscheme $X'\subseteq \AA^r$ with projective closure $\overline{X'}\subseteq \PP^r$ we define $\deghat(X'):=\deghat(\overline{X'})$.  
\end{defn}
This provides a notion of degree which ignores nonreduced structure but takes into account components of lower dimension. Similar definitions have appeared in the literature: for instance, in the language of~\cite[\S3]{bayer-mumford}, we would have $\deghat(X) = \sum_{j=0}^{\dim X} \text{geom-deg}_j (X)$.
This definition is useful when bounding the number of points of a scheme over a finite field.

\begin{lemma}
Let $X\subseteq \PP^r_{\FF_q}$ be a closed subscheme, where $\FF_q$ is a finite field.  Then
\[
\#X(\FF_q) \leq \deghat(X) q^{\dim X}.
\]
\end{lemma}
\begin{proof}
This is an immediate consequence of the Schwarz-Zippel lemma, which implies that for an irreducible algebraic variety $V_i\subseteq \PP^r_{\FF_q}$ we have $\#V_i(\FF_q)\leq \deg(V_i)q^{\dim V_i}$.
\end{proof}

\begin{lemma}\label{lem:bezout}
Let $\kk$ be any field and let $X\subseteq \mathbb A^r_{\kk}$.   Let $f_0,f_1,\dots,f_t$ by polynomials in $\kk[x_1,\dots,x_r]$.  If $X'=X\cap \VV(f_0,f_1,\dots,f_t)$, then 
$
\deghat( X') \leq \deghat(X) \cdot \prod_{i=0}^t \deg(f_i).
$
\end{lemma}
\begin{proof}
This follows from the refined version of Bezout's Theorem~\cite[Example~12.3.1]{fulton}.
\end{proof}

\section{A uniform lower bound on Hilbert functions}\label{sec:Hilbert function}
For a subscheme of $\PP^r$, the Hilbert function in degree $d$ is controlled by the Hilbert polynomial, at least if $d$ is very large related to some invariants of the subscheme.  We analyze the Hilbert function of a subscheme at the opposite extreme, where the degree of the subscheme is much larger than $d$.  The following lemma, which applies to subschemes of arbitrarily high degree, provides uniform lower bounds that are crucial to bounding the error in our sieves.
\begin{lemma}\label{lem:key Hilbert function}
Let $\kk$ be an arbitrary field and fix some $e\geq 0$. Let $V\subseteq \PP^r_{\kk}$ be any closed, $m$-dimensional subscheme of degree $>e$ with homogeneous coordinate ring $R$.  
\begin{enumerate}
	\item  We have $\dim R_d \geq h^0(\PP^{m},\cO_{\PP^{m}}(d))$ for all $d$.
	\item  For any $0<\epsilon<1$, there exists a constant $C$ depending only on $e,m$ and $\epsilon$ such that
\[
\dim R_d > (e+\epsilon)\cdot h^0(\PP^{m},\cO_{\PP^{m}}(d)) 
\]
for all $d\geq Ce^{m+1}$.
	\end{enumerate}
\end{lemma}
\begin{proof}
We can assume that we are working over an infinite field, since this will not change the values of the Hilbert function of $R$. For part (1), we simply take a linear Noether normalization $\kk[t_0,t_1,\dots,t_m]\subseteq R$ of the ring $R$~\cite[Theorem~13.3]{eisenbud}. This yields $k[t_0,t_1,\dots,t_m]_d\subseteq R_d$, giving the statement about Hilbert functions.

We prove part (2) of the theorem by induction on $m$. Let $S=\kk[x_0,x_1,\dots,x_r]$ and let  $I_V\subseteq S$ be the saturated, homogeneous ideal defining $V$. Thus $R=S/I_V$.  If $m=0$, then we have
$
\dim R_d \geq \min\{ d+1,\deg V\} \geq \min\{d+1,e+1\}$ which is at least $e+\epsilon$ for all $d\geq e$.  This proves the case $m=0$, where the constant $C$ can be chosen to be $1$.

Now assume the claim holds for all closed subschemes of dimension less than $m$. Let $V\subset\PP^r_{\kk}$ be a closed subscheme with $\dim V=m\geq1$. Fix $0<\epsilon<1$. Since we are working over an infinite field, ~\cite[Lemma 13.2(c)]{eisenbud} allows us to choose a linear form $\ell$ that is a nonzero divisor on $R$. Hence we have a short exact sequence:
\begin{equation}\label{eqn:exact sequence}
\begin{tikzcd}[column sep = 2.5em]
0\rar& R(-1)\rar{\cdot\ell}& R \rar & R/\ell \rar& 0.
\end{tikzcd}
\end{equation}
Letting $W=V\cap V(\ell)$ we know that $\dim W=m-1$ and $\deg W=\deg V$. Moreover, If $I_V$ is the saturated ideal defining $V$ and if $I_W$ is the saturated ideal defining $W$, then we have that $I_W$ is the saturation of $I_V+\langle \ell\rangle$ at $\mathfrak m$. In particular, since $I_W$ contains $I_V+\langle \ell\rangle$, we have
\[
\dim (S/(I_V + \langle \ell\rangle))_i\geq\dim (S/I_W)_i.
\]
By induction, for $\epsilon'$ with $0<\epsilon<\epsilon'<1$, there exists $C'$ depending on $\epsilon',e$ and $m-1$ where
\[
\dim (S/I_W)_i\geq (e+\epsilon')\binom{m-1+i}{m-1}
\]
for all $i\geq C'e^{m}$.  Iteratively applying the exact sequence \eqref{eqn:exact sequence} for $d\geq C'e^m$ we obtain:
\begin{align*}
\dim R_d &\geq  \dim R_{C'e^{m}} + \sum_{i=C'e^{m-1}}^{d} \dim (S/I_V+\ell)_{i} \\
&\geq  \dim R_{C'e^{m}} + \sum_{i=C'e^{m-1}}^{d} \dim (S/I_W)_{i} \\
&\geq \sum_{i=C'e^{m}}^{d} (e+\epsilon')\binom{m-1+i}{m-1}.
\end{align*}
The identity $\sum_{i=a}^b \binom{i+k}{k} = \binom{b+k+1}{k+1} - \binom{a+k}{k+1}$ implies  that $\sum_{i=C'e^{m}}^{d} (e+\epsilon')\binom{m-1+i}{m-1}$ can be rewritten as $(e+\epsilon')\left( \binom{m+d}{m} - \binom{m-1+C'e^{m}}{m}\right).$  There exists a constant $C$ depending on $e$, $\epsilon$, and $m$ so that $(\epsilon'-\epsilon)\binom{m+d}{m}\geq (e+\epsilon')\binom{m-1+C'e^{m}}{m}$ for all $d\geq Ce^{m+1}$.  Thus, for all $d\geq Ce^{m+1}$ we have
\[
 \dim R_d\geq (e+\epsilon')\binom{m+d}{m}-(\epsilon'-\epsilon)\binom{m+d}{d}=(e+\epsilon)\binom{m+d}{m}.\qedhere
 \]
\end{proof}
\begin{remark}
Asymptotically in $e$, the bound of $Ce^{2}$ is the best possible for curves. For instance, let $C\subseteq \PP^r$ be a curve of degree $(e+1)$ lying inside some plane $\PP^2\subseteq \PP^r$. Let $R$ be the homogeneous coordinate ring of $C$. If $d\geq e$ then the Hilbert function is given by
\[
\dim R_d = (e+1)d -\tfrac{e^2-e}{2}.
\]
Thus, if we want $\dim R_d\geq (e+\epsilon)(d+1)$, we will need to let $d\geq \frac{e^2+3e+2\epsilon}{2(1-\epsilon)}\approx \frac{1}{2}e^2$.  It would be interesting to know if the bound $Ce^{m+1}$ is the best possible for higher dimensional varieties.
\end{remark}

\section{Geometric Analysis}\label{sec:geometric}
In this section we analyze the 
 geometric picture for the distribution of parameters on $X$.
The basic idea behind the proof of Theorem~A is that $f_0, f_1, \dots,f_k$ fail to be parameters on $X$ if and only if they vanish along some $(n-k)$-dimensional subvariety of $X$. Since the Hilbert polynomial of any $(n-k)$-dimensional variety grows like $d^{n-k}$, when we restrict a degree $d$ polynomial $f_j$ to such a subvariety, it can be written in terms of  $\approx d^{n-k}$ distinct monomials. The polynomial $f_j$ will vanish along the subvariety if and only if all of the $\approx d^{n-k}$ coefficients vanish. This rough estimate explains the growth of the codimension of $\Par_{d,k}(X)$ as $d\to \infty$.  

We begin by constructing the schemes $\Par_{k,d}(X)$. 
Fix $X\subseteq \PP^r_{\kk}$ a closed subscheme of dimension $n$ over a field $\kk$.  
Given $k<n$ and $d>0$, let $\mathscr A_{k,d}$ be the affine space $H^0(\PP^r,\cO_{\PP^r}(d))^{\oplus k+1}$ and $\kk[c_{0,1},\dots,c_{k,\binom{r+d}{d}}]$ be the corresponding polynomial ring.  We enumerate the monomials in $H^0(\PP^r,\cO_{\PP^r}(d))$ as $m_1,\dots,m_{\binom{r+d}{d}}$, and then define the universal polynomial
\[
F_i := \sum_{j=1}^N c_{i,j}m_j \in \kk[c_{0,1},\dots,c_{k,\binom{r+d}{d}}] \otimes_{\kk} \kk[x_0,\dots,x_r].
\]
Given a closed point $c\in \mathscr A_{k,d}$ we can specialize $F_0,\dots,F_k$ and obtain polynomials $f_0,\dots, f_k\in \kappa(c)[x_0,\dots,x_r]$, where $\kappa(c)$ is the residue field of $c$. We will thus identify each element of $\mathscr A_{k,d}(\kk)$ with a collection of polynomials $\ff=(f_0,f_1,\dots,f_k)\in \kk[x_0,\dots,x_r]$.  

Now define $\Sigma_{k,d}(X)\subseteq X\times \mathscr A_{k,d}$ via the equations $F_0,\dots,F_k$. Consider the second projection $p_2: \Sigma^{(k,d,X)}\to \mathscr A_{k,d}$. Given a point $\ff=(f_0,\dots,f_k)\in \mathscr A_{k,d}$, the fiber $p_2^{-1}(\ff)\subseteq X$ can be identified with the points lying in $X\cap \VV(f_0,\dots,f_k)$. For generic choices of $\ff$ (after passing to an infinite field if necessary) the polynomials $(f_0,\dots,f_k)$ will have codimension $k+1$, and thus the fiber $p_2^{-1}(\ff)$ will have dimension $n-k-1$.

There is a closed sublocus in $\Par_{k,d}(X)\subsetneq \mathscr A_{k,d}$ where the dimension of the fiber is at least $n-k$, and we give $\Par_{k,d}(X)$ the reduced scheme structure. It follows that $\Par_{k,d}(X)$ parametrizes collections $\ff=(f_0,\dots,f_k)$ of degree $d$ polynomials which fail to be parameters on $X$.

\begin{remark}\label{rmk:Par ZZ}
If we fix $X_{\ZZ}\subseteq \PP^r_{\ZZ}$, then we can follow the same construction to obtain a scheme $\Par_{k,d}(X_{\ZZ})\subseteq \mathscr A_{k,d}$. Writing $X_{\kk}$ as the pullback $X\times_{\Spec \ZZ}\Spec \kk$, we observe that the equations defining $\Sigma_{k,d}(X_{\kk})$ are obtained by pulling back the equations defining $\Sigma_{k,d}(X_{\ZZ})$.  It follows that $\Par_{k,d}(X_\ZZ)\times_{\Spec \ZZ} \Spec(\kk)$ has the same set-theoretic support as $\Par_{k,d}(X_\kk)$.  
\end{remark}

\begin{defn}
We let $\Par_{k,d}^{\bad}(X)$ be the locus of points in $\Par_{k,d}(X)$ where $f_0,\dots,f_{k-1}$ already fail to be parameters on $X$ and let $\Par^\good_{k,d}(X):=\Par_{k,d}(X)\setminus \Par^{\bad}_{k,d}(X).$ We set $\Par^\bad_{0,d}(X)=\varnothing$.
\end{defn}

\begin{remark}\label{rmk:splitting}
We have a splitting:
\begin{align*}
\mathscr A_{k,d} &\to \mathscr A_{k-1,d} \times \mathscr A_{0,d}\\
(f_0,\dots,f_{k})&\mapsto ((f_0,\dots,f_{k-1}),f_{k}).
\end{align*}
Letting $\pi: \Par_{k,d}(X)\to \mathscr A_{k-1,d}$ be the induced projection, we will speak about the degree of the image of $\pi$ (considered in $\mathscr A_{k-1,d}$) and the degree of a fiber of $\pi$  (considered in $\mathscr A_{0,d}$). 
\end{remark}

\begin{proof}[Proof of Theorem~A]
First consider the case $k=n$. There is a natural rational map from $\mathscr A_{n,d}$ to the Grassmanian $\Gr(n+1, S_d)$ given by sending the polynomials $(f_0,\dots,f_{n})$ to the linear space that they span. Inside of the Grassmanian, the locus of choices of $(f_0,\dots,f_{n})$ that all vanish on a point of $X$ is a divisor in the Grassmanian defined by the Chow form; see~\cite[3.2.B]{gkz}. The preimage of this hypersurface in $\mathscr A_{n,d}$ is a hypersurface contained in $\Par_{n,d}(X)$, and thus $\Par_{n,d}(X)$ has codimension $1$.

For $k<n$, we will induct on $k$. Let $k=0$. A polynomial $f_0$ will fail to be a parameter on $X$ if and only if $\dim X = \dim (X\cap \VV(f_0))$. This happens if and only if $f_0$ is a zero divisor on a top-dimensional component of $X$. Let $V$ be the reduced subscheme of some top-dimensional irreducible component of $X$ and let $\mathcal I_V$ be the defining ideal sheaf of $V$. Then the set of zero divisors of degree $d$ on $V$ will form a linear subspace in $\mathscr A_{0,d}$ corresponding to the elements of the vector subspace $H^0(\mathcal I_V(d))$. The codimension of $H^0(\mathcal I_V(d))\subseteq S_d$ is precisely given by the Hilbert function of the homogeneous coordinate ring of $V$ in degree $d$. By applying Lemma~\ref{lem:key Hilbert function}(1), we conclude that for all $d$ this linear space has codimension at least $\binom{n+d}{d}$. Since $\Par_{0,d}(X)$ is the union of these linear spaces over all top-dimensional components of $X$, this proves that $\codim \Par_{0,d}(X)\geq \binom{n+d}{d}$.

Take the induction hypothesis that we have proven the statement for $\Par_{j,d}(X')$ for all $X'\subseteq \PP^r$ and all $j\leq k-1$. We separate $\Par_{k,d}(X)=\Par^\bad_{k,d}(X)\sqcup \Par^\good_{k,d}(X)$ and will show that each locus has sufficiently large codimension. We begin with $\Par^\bad_{k,d}(X)$.  By definition, the projection $\pi$ from Remark~\ref{rmk:splitting} maps $\Par^\bad_{k,d}(X)$  onto $\Par_{k-1,d}(X)$.  We thus have:
\[
\codim(\Par^\bad_{k,d}(X), \mathscr A_{k,d})\geq \codim(\Par_{k-1,d}(X),\mathscr A_{k-1,d})\geq  \binom{n-k+1+d}{n-k+1} \geq \binom{n-k+d}{n-k},
\]
where the middle inequality follows by induction.

Now consider an arbitrary point $\ff=(f_0,\dots,f_{k})$ in $\Par^\good_{k,d}(X)$. By definition, $f_0,\dots,f_{k-1}$ must be parameters on $X$, and thus $\pi(\ff) \in \mathscr A_{k-1,d} \setminus \Par_{k-1,d}(X)$. Using the splitting of Remark~\ref{rmk:splitting}, the fiber of $\Par^\good_{k,d}(X)$ over $\ff$ can be identified with $\Par_{0,d}(X')$ where $X':=X\cap \VV(f_0,\dots,f_{k-1})$.  Since $(f_0,\dots,f_{k-1})\notin\Par_{k-1,d}(X)$, we have that $\dim X'=n-k$. The inductive hypothesis thus guarantees that
$
\codim \Par_{0,d}(X') \geq \binom{\dim X'+d}{d} = \binom{n-k+d}{d}.
$
\end{proof}

\section{Probabilistic Analysis I: Proof of Theorem~B}\label{sec:probabilistic}

Throughout this section, we let $X \subseteq \PP^r_{\FF_q}$ be a projective scheme of dimension $n$ over a finite field $\FF_q$. Recall that $S_d=H^0(\PP^r, \cO_{\PP^r}(d))$.  We define
\[
\cP_{d,k}=\left\{\begin{matrix}(f_0, f_1, \dots,f_{k}) \text{ that }\\ \text{ are parameters on $X$}\end{matrix}\right\}\subset S_{d}^{k+1}.
\]
In Theorem~B, we compute the following limit (which a priori might not exist):
\[
\lim_{d\to \infty}\Prob\left(\begin{matrix}(f_0, f_1,\dots,f_{k}) \text{ of degree $d$ } \\ \text{ are parameters on $X$}\end{matrix}\right):=\lim_{d\to \infty}\frac{\#\cP_{d,k}}{\#S_{d}^{k+1}}.
\]
As in the geometric case, there is a bifurcation between when $k=n$ and $k<n$. The case $k=n$ largely parallels Poonen's work~\cite{poonen}: we will sieve over closed points of $X$ and show that the asymptotic probability that $(f_0, f_1 \dots f_n)$ are parameters on $X$ equals the product of local probabilities at the closed points of $X$, and the resulting formula will correspond with the value of a zeta function. This approach was already worked out by Bucur and Kedlaya in~\cite{bucur-kedlaya}.  They assume that $X$ is smooth, but their proof does not need that assumption.

When $k<n$ we need to significantly alter Poonen's sieve. In this section, we focus on proving that the asymptotic probability converges to $1$ as $d\to \infty$. For this, we will use a coarse error bound based on the geometric picture developed in \S\ref{sec:geometric}. In \S\ref{sec:error}, we provide a deeper analysis of the limit probability based on the detailed geometry of $X$.

\begin{prop}\label{prop:precise estimate}
If $k<n$ then
\[
\Prob\left(\begin{matrix}(f_0,f_1,\dots,f_{k}) \text{ of degree $d$ } \\ \text{ are parameters on $X$}\end{matrix}\right)\geq 1 - \deghat(X)(1+d+d^2+\dots+d^k)q^{-\binom{n-k+d}{n-k}}.
\]
\end{prop}
\begin{proof}
We induct on $k$ and largely follow the structure of the proof of Theorem~A. First, let $k=0$. A polynomial $f_0$ will fail to be a parameter on $X$ if and only if it is a zero divisor on a top-dimensional component $V$ of $X$. There are at most $\deghat(X)$ many such components. As argued in the proof of Theorem~A, the set of zero divisors on $V$ corresponds to the elements of $H^0(\PP^r,\mathcal I_V(d))$ which has codimension at least $\binom{n+d}{d}$ in $S_d$. It follows that
\[
\Prob\left(\begin{matrix}f_0  \text{ of degree $d$ is } \\ \text{ not a parameter on $X$}\end{matrix}\right)\leq \deghat(X)q^{-\binom{n+d}{d}}.
\]

Now consider the induction step. We will separately compute the probability that $\ff=(f_0,f_1,\dots,f_k)$ lies in $\Par^\bad_{k,d}(X)$ and the probability that $\ff$ lies in $\Par^\good_{k,d}(X)$.  By definition, the projection $\pi$ maps $\Par^\bad_{k,d}(X)$  onto $\Par_{k-1,d}(X)$, and by induction
\begin{align*}
\Prob\bigl(\pi(\ff)\in \mathscr P_{k-1,d}(X)(\FF_q)\bigr)&\leq \deghat(X)\left(1+d+d^2+\dots+d^{k-1}\right)q^{- \binom{n-k+1+d}{n-k+1}}\\
& \leq \deghat(X)\left(1+d+d^2+\dots+d^{k-1}\right)q^{- \binom{n-k+d}{n-k}}.
\end{align*}
We now assume $\ff\notin \Par^\bad_{k,d}(X)$. We thus have that $f_0,\dots,f_{k-1}$ are parameters on $X$. As in the proof of Theorem~A, the fiber $\pi^{-1}(\ff)$ can be identified with $\Par_{0,d}(X')$ where $X':=X\cap \VV(f_0,f_1,\dots,f_{k-1})$. By construction $\dim X'=n-k$ and by Lemma~\ref{lem:bezout}, $\deghat(X') \leq \deghat(X) \cdot d^k$. Our inductive hypothesis thus implies that
\[
\Prob\left(\begin{matrix}(f_0,\dots,f_k)\in \Par_{k,d}(X)(\FF_q)\text{ given} \\ \text{ that } (f_0,\dots,f_{k-1})\notin \Par_{k-1,d}(X)(\FF_q)\end{matrix} \right) \leq \deghat (X')q^{-\binom{n-k+d}{n-k}} \leq  \deghat(X) \cdot d^kq^{-\binom{n-k+d}{n-k}}.
\]
Combining the estimates for $\Par^\bad_{k,d}(X)$ and $\Par^\good_{k,d}(X)$ yields the proposition.
\end{proof}
\begin{proof}[Proof of Theorem~B] \ 

If $k<n$, then we apply Proposition~\ref{prop:precise estimate} to obtain
\begin{align*}
\lim_{d\to \infty} \Prob\left(\begin{matrix}(f_0,\dots,f_{k}) \text{ of degree $d$ } \\ \text{ are parameters on $X$}\end{matrix}\right)& \geq\lim_{d\to \infty} 1 - \deghat(X)\left(d^0+d^1+\dots+d^k\right)q^{-\binom{n-k+d}{n-k}}= 1.
\end{align*}

Now let $k=n$.  For completeness, we summarize the proof of \cite[Theorem~1.2]{bucur-kedlaya}. We fix $e$, which will go to $\infty$, and separate the argument into low, medium, and high degree cases.

\paragraph{{\em Low degree argument}}
For a zero dimensional subscheme $Y$, we have that $S_d$ surjects on $H^0(Y,\cO_Y(d))$ when $d\geq \deg Y-1$~\cite[Lemma~2.1]{poonen}.  So if $d>\deg P-1$, the probability that $f_0,f_1,\dots,f_n$ all vanish at a closed point $P\in X$ is $1 - q^{-(n+1)\deg P}$.  If $Y\subseteq X$ is the union of all points of degree $\leq e$, and if $d\geq \deg Y-1$, then the surjection onto $H^0(Y,\cO_Y(d))$ implies that the probabilities at the points $P\in Y$ behave independently.  This yields:
\[
\Prob\left(\begin{matrix}f_0,f_1,\dots,f_{n} \text{ of degree $d$ are} \\ \text{ parameters on $X$ at all points }\\ \text{$P\in X$ where $\deg(P)\leq e$}\end{matrix}\right) = \prod_{\overset{P\in X}{\deg(P)\leq e}} 1 - q^{-(n+1)\deg P}.
\]
\paragraph{{\em Medium degree argument}}
Our argument is nearly identical to~\cite[Lemma~2.4]{poonen}, and covers all points whose degree lies in the range $[e, \frac{d}{n+1}]$. For any such point $P\in X$, $S_d$ surjects onto $H^0(P,\cO_P(d))$ and thus the probability that $(f_0,f_1,\dots,f_n)$ all vanish at $P$ is $q^{-\ell(n+1)}$.  By~\cite{lang-weil}, $\#X(\FF_{q^\ell})\leq Kq^{\ell n}$ for some constant $K$ independent of $\ell$.  We have\begin{align*}
\Prob\left(\begin{matrix}f_0,f_1,\dots,f_{n} \text{ of degree $d$ all} \\ \text{ vanish at some $P\in X$ }\\ \text{ where $e<\deg(P)\leq \left\lfloor \frac{d}{n+1}\right\rfloor$}\end{matrix}\right)
\leq \sum_{\ell=e}^{ \left\lfloor \tfrac{d}{n+1}\right\rfloor} \#X(\FF_{q^{\ell}}) q^{-\ell(n+1)}
&\leq \sum_{\ell=e}^{\infty} Kq^{\ell n} q^{-(n+1)\ell}= \frac{Kq^{-e}}{1-q^{-1}}.
\end{align*}
This tends to $0$ as $e\to \infty$, and therefore does not contribute to the asymptotic limit.

\paragraph{{\em High degree argument}}
By the case when $k=n-1$, we may assume that $f_0,f_1,\dots,f_{n-1}$ form a system of parameters with probability $1-o(1)$. So we let $V$ be one of the irreducible components of this intersection (over $\FF_q$) and we let $R$ be its homogeneous coordinate ring. If $\deg V\leq \frac{d}{n+1}$, then it can be ignored as we considered such points in the low and medium degree cases. Hence, we can assume $\deg V >\frac{d}{n+1}$.  Since
$
\dim R_\ell \geq \min\{\ell+1,\deg R\}
$
for all $\ell$, the probability that $f_n$ vanishes along $V$ is at most $q^{-\lfloor \frac{d}{n+1}\rfloor -1}$. Hence the probability of vanishing on some high degree point is bounded by $O(d^nq^{-\lfloor \frac{d}{n+1}\rfloor -1})$ which is $o(1)$ as $d \to \infty$.

Combining the various parts as $e\to \infty$, we see that the low degree argument converges to $\zeta_X(n+1)^{-1}$ and the contributions from the medium and high degree points go to $0$.
\end{proof}

\begin{remark}
It might be interesting to consider variants of Theorem~B that allow imposing conditions along closed subschemes, similar to Poonen's Bertini with Taylor Coefficients~\cite[Theorem~1.2]{poonen}.  For instance, \cite[Theorem 1]{kedlaya-more-etale} might be provable by such an approach, though this would be more complicated than the original proof.
\end{remark}

Proposition~\ref{prop:precise estimate} also yields an effective bound on the degree of a full system of parameters over a finite field.  Sharper bounds can obtained if one allows the $f_i$ to have different degrees.  
\begin{cor}\label{cor:effective Noether normalization}
\
 \begin{enumerate}
	\item  If $d_1$ satisfies
$d_1^{n-1}q^{-d_1-1} < (n\cdot \deghat (X))^{-1}$, then there exist $g_0, g_1,\dots,g_{n-1}$ of degree $d_1$ that are parameters on $X$.
	\item  Let $X'$ be $0$-dimensional.  If $\max\{d_2+1,q\}\geq \deghat(X')$ then there exists a degree $d_2$ parameter on $X'$.
\end{enumerate}
\end{cor}
\begin{proof}
Applying Proposition~\ref{prop:precise estimate} in the case $k=n-1$ yields (1).  For (2), let $f$ be a random degree $d$ polynomial and let $P\in X'$ be a closed point.  Since the dimension of the image of $S_d$ in $H^0(P,\cO_P(d))$ is at least $\min\{d+1,\deg P\}$, the probability that $f$ vanishes at $P$ is at worst $q^{-\min\{d+1,\deg P\}}$ which is at least $q^{-1}$.  It follows that the probability that a degree $d$ function vanishes on some point of $X'$ is at worst $\sum_{P\in X'} q^{-1}\leq \deghat(X')q^{-1}$. Thus if $q>\deghat(X')$, this happens with probability strictly less than $1$.  On the other hand, if $d+1\geq \deghat(X')$ then polynomials of degree $d$ surject onto $H^0(X',\cO_{X'}(d))$ and hence we can find a parameter on $X'$ by choosing a polynomial that restricts to a unit on $X'$.
%
\end{proof}

\begin{proof}[Proof of Theorem~C]
If $\dim X=0$, then we can directly apply Corollary~\ref{cor:effective Noether normalization}(2) to find a parameter of degree $d=\max\{d,d^n\}$.  So we assume $n:=\dim X>0$.  Since $d>\log_q \deghat(X) + \log_q n+ n\log_q d$ it follows that 
$
(n\cdot \deghat(X))^{-1}>q^{-d}d^n>q^{-d-1}d^{n-1}.
$
Applying Corollary~\ref{cor:effective Noether normalization}(1), we find $g_0,g_1,\dots,g_{n-1}$ in degree $d$ that are parameters on $X$.  Let $X'=X\cap V(g_0,g_1,\dots,g_{n-1})$.  Since $\max\left\{d,\frac{q}{d^{n}}\right\}\geq \deghat(X)$ it follows that 
$\max\left\{d^{n+1},q\right\}\geq d^{n}\deghat(X)\geq \deghat(X')$, and Corollary~\ref{cor:effective Noether normalization}(2) yields a parameter $g_n$ of degree $d^{n+1}$ on $X'$.  Thus $g_0^{d^n}, g_1^{d^n},\dots, g_{n-1}^{d^n}, g_n$ are parameters of degree $d^{n+1}$ on $X$.

\end{proof}

\section{Probabilistic Analysis II: The Error Term}\label{sec:error}
In this section, we let $k<n$ and we analyze the error terms in Theorem~B more precisely. In particular, we will show that the probabilities are controlled by the probability of vanishing along an $(n-k)$-dimensional subvariety, with varieties of lowest degree contributing the most.
Theorem~\ref{thm:main finite field} is the estimate obtained by tracking subvarieties of degree $\leq e$.

\begin{notation}\label{notation:comps and equidim}
For a subscheme $Z\subseteq X$, we write $|Z|$ for the number of irreducible components of $Z$, and we write $\dim Z \equiv k$ if $Z$ is equidimensional of dimension $k$. 
\end{notation}

\begin{thm}\label{thm:main finite field}
Let $X\subseteq \PP^r_{\FF_q}$ be a projective scheme of dimension $n$. Fix $e$ and let $k<n$. The probability that random polynomials $f_0,f_1,\dots,f_k$ of degree $d$ are parameters on $X$ is
\[
\Prob\left(\begin{matrix}f_0,f_1,\dots,f_{k} \text{ of degree $d$ } \\ \text{ are parameters on $X$}\end{matrix}\right) = 1 \ - 
\sum_{\begin{smallmatrix}Z\subseteq X \text{reduced} \\ \dim Z \equiv n-k\\ \deg Z \leq e  \end{smallmatrix}}(-1)^{|Z|-1}q^{-(k+1)h^0(Z,\cO_Z(d))}+ o\left(q^{-e(k+1)\binom{n-k+d}{n-k}}\right).
\]
\end{thm}

The terms of the above sum have the form $q^{-(k+1)h^0(Z,\cO_Z(d))}$, where $Z\subseteq X$ is an $(n-k)$-dimensional subvariety. Hence, the exponents are controlled by the Hilbert polynomial of a $(n-k)$-dimensional variety, and will grow like $d^{n-k}$, converging to $0$ rapidly as $d\to \infty$. 

Our proof of Theorem~\ref{thm:main finite field} adapts Poonen's sieve in a couple of key ways. The first big difference is that instead of sieving over closed points, we will sieve over $(n-k)$-dimensional subvarieties of $X$; this is because polynomials $(f_0,\dots,f_k)$ will fail to be parameters on $X$ only if they vanish along some $(n-k)$-dimensional subvariety.  

The second difference is that the resulting probability formula will not be a product of local factors. This is because the values of a function can never be totally independent along two higher dimensional varieties with a nontrivial intersection. For instance, Lemma~\ref{local_probability} shows that the probability that a degree $d$ polynomial vanishes along a line is $q^{-(d+1)}$, but the probability of vanishing along two lines that intersect in a point is $q^{-(2d+1)}>(q^{-(d+1)})^2$.  

The following result characterizes the individual probabilities arising in our sieve.

\begin{lemma}\label{local_probability}
If $Z\subseteq \PP_{\FF_{q}}^r$ is a reduced, projective scheme over a finite field $\FF_q$ with homogeneous coordinate ring $R$ then
\[
 \Prob\left(\begin{matrix}(f_0,\dots,f_{k}) \text{ of degree $d$ } \\ \text{ vanish along $Z$}\end{matrix}\right)=\left(\frac{1}{\#R_d}\right)^{k+1}.
 \]
If $d$ is at least the Castelnuovo-Mumford regularity of the ideal sheaf of $Z$, then
\[
 \Prob\left(\begin{matrix}(f_0,\dots,f_{k}) \text{ of degree $d$ } \\ \text{ vanish along $Z$}\end{matrix}\right)=q^{-(k+1)h^0(Z, \cO_Z(d))}.
 \]
\end{lemma}

\begin{proof}
Let $I\subseteq S$ be the homogeneous ideal defining $Z$, so that $R=S/I$.
An element $h\in S_d$ vanishes along $Z$ if and only if it restricts to $0$ in $R_d$ i.e. if and only if it lies in $I_d$. Since we have an exact sequence of $\FF_q$-vector spaces:
\[
\begin{tikzcd}[column sep =2.5em]
0\rar& I_d\rar &S_d\rar &R_d\rar& 0
\end{tikzcd}
\]
we obtain
\[
\Prob(h \text{ vanishes on } Z) = \frac{\# I_d}{\# S_d} = \frac{1}{\# R_d}.
\]
For $k+1$ elements of $S_d$, the probabilities of vanishing along $Z$ are indepdenent and this yields the first statement of the lemma.

We write $\widetilde{I}$ for the ideal sheaf of $Z$.  If $d$ is at least the regularity $\widetilde{I}$, then $H^1(\PP_{\FF_q}^r, \widetilde{I}(d))=0$. Hence there is a natural isomorphism between $R_d$ and $H^0(Z,\cO_Z(d))$.  Thus, we have
\[
\frac{1}{\# R_d} = q^{-h^0(Z,\cO_Z(d))},
\]
yielding the second statement.
\end{proof}

\begin{proof}[Proof of Theorem~\ref{thm:main finite field}]
Throughout the proof, we set $\epsilon_{e,k}$ to be the error term for a given $e$ and $k$, namely
$
\epsilon_{e,k} := q^{-e(k+1)\binom{n-k+d}{n-k}}.
$
We also set:
\begin{align*}
\cP_{d,k}&:=\left\{\begin{matrix}(f_0,f_1,\dots,f_{k})\\ \text{ are parameters on $X$}\end{matrix}\right\}\\
\cL_{d,k,e}&:= \left\{\begin{matrix}(f_0,f_1,\dots,f_{k}) \text{ vanish along a variety $Z$}\\ \text{ where $\dim Z=(n-k)$ and $\deg(Z)\leq e$}\end{matrix}\right\}\\
\cM_{d,k,e}&:= \left\{\begin{matrix}(f_0,f_1,\dots,f_{k})\notin \cL_{d,k,e} \text{ which vanish along a variety $Z$}\\ \text{ where $\dim Z=(n-k)$ and $e<\deg(Z)\leq e(k+1)$}\end{matrix}\right\}\\
\cH_{d,k,e}&:= \left\{\begin{matrix}(f_0,f_1,\dots,f_{k})\notin \cL_{d,ke}\cup \cM_{d,k,e} \text{ which vanish along a variety $Z$ }\\ \text{ where $\dim Z=(n-k)$ and $e(k+1)<deg(Z)$}\end{matrix}\right\}.\\
\end{align*}
Note that if $(f_0,f_1,\dots,f_k)$ vanish along a variety of dimension $>n-k$ then they will also vanish along a high degree variety, and hence we do not need to count this case separately.
For $\ff=(f_0,f_1,\dots,f_k)\in S_d^{k+1}$, we thus have
\begin{align*}
\Prob(\ff\in \cP_{d,k})&= 1-\Prob(\ff\in \cL_{d,k,e}\cup \cM_{d,k,e}\cup \cH_{d,k,e}) \\
&=1-\Prob(\ff\in \cL_{d,k,e})-\Prob(\ff\in \cM_{d,k,e})-\Prob(\ff\in \cH_{d,k,e}).
\end{align*}

It thus suffices to show that
\begin{align*}
\Prob(\ff\in \cL_{d,k,e})=\sum_{\begin{smallmatrix}Z\subseteq X \text{reduced} \\ \dim Z \equiv n-k\\ \deg Z \leq e  \end{smallmatrix}}(-1)^{|Z|-1}q^{-(k+1)h^0(Z,\cO_Z(d))} + o\left(\epsilon_{e,k}\right)
\end{align*}
and that 
$\Prob(\cM_{d,k,e})$ and $\Prob(\cH_{d,k,e})$ are each in $o(\epsilon_{e,k})$.

We proceed by induction on $k$.  When $k=0$ the condition that $f_0$ is a parameter on $X$ is equivalent to $f_0$ not vanishing along a top-dimensional component of $X$. Thus, combining Lemma \ref{local_probability} with an inclusion/exclusion argument implies the exact result:
\[
\Prob(f_0\in \cP_{d,0})=1-\sum_{\begin{smallmatrix}Z\subseteq X \text{reduced} \\ \dim Z \equiv n-k  \end{smallmatrix}} (-1)^{|Z|-1}q^{-h^0(Z, \cO_Z(d))}.
\]
By basic properties of the Hilbert polynomial, as $d\to \infty$ we have
\[
h^{0}(Z, \cO_Z(d))= \frac{\deg(Z)}{n!}d^{n}+ o(d^n)=\deg(Z) \binom{n+d}{d} + o(d^n).
\]
Hence for the fixed degree bound $e$, we obtain:
\begin{align*}
\Prob(\cP_{d,0})&=1-\sum_{\begin{smallmatrix}Z\subseteq X \text{reduced} \\ \dim Z \equiv n-k\\ \deg Z\leq e  \end{smallmatrix}} (-1)^{|Z|-1}q^{-h^0(Z, \cO_Z(d))} -\sum_{\begin{smallmatrix}Z\subseteq X \text{reduced} \\ \dim Z \equiv n-k\\ \deg Z> e  \end{smallmatrix}} (-1)^{|Z|-1}q^{-h^0(Z, \cO_Z(d))}\\
&= 1-\sum_{\begin{smallmatrix}Z\subseteq X \text{reduced} \\ \dim Z \equiv n-k\\ \deg Z\leq e  \end{smallmatrix}} (-1)^{|Z|-1}q^{-h^0(Z, \cO_Z(d))} + o(\epsilon_{e,0}).
\end{align*}

We now consider the induction step. Let $\ff=(f_0,\dots,f_k)$ drawn randomly from $S_d^{k+1}$.  Here we separate into low, medium, and high degree cases.

\paragraph{{\em Low degree argument}}
Let $\mathbf V_{k,e}$ denote the set of integral projective varieties $V\subseteq X$ of dimension $n-k$ and degree $\leq e$. We have $\ff\in \cL_{d,k,e}$ if and only if $\ff$ vanishes on some $V\in \mathbf V_{k,e}$.
Since $\mathbf V_{k,e}$ is a finite set, we may use an inclusion-exclusion argument to get \begin{align*}
\Prob(\ff\in \cL_{d,k,e})
&=\sum_{\begin{smallmatrix}Z\subseteq X \text{ a union of} \\ V\in \mathbf V_{k,e}  \end{smallmatrix}}(-1)^{|Z|-1}\Prob\left(\begin{matrix}f_0,\dots,f_{k} \text{ of degree $d$ } \\ \text{ vanish along $Z$}\end{matrix}\right).
\intertext{If $\deg Z>e$ then Lemma \ref{local_probability} implies that those terms can be absorbed into the error term $o(\epsilon_{e,k})$. Moreover, assuming that $Z$ is a union of $V\in \mathbf{V}_{k,e}$ satisfying $\deg(Z)\leq e$ is equivalent to assuming $Z$ is reduced and equidmensional of dimensional $n-k$.  We thus have:}
&=\sum_{\begin{smallmatrix}Z\subseteq X \text{ reduced} \\ \dim Z \equiv n-k\\ \deg Z \leq e  \end{smallmatrix}}(-1)^{|Z|-1}\Prob\left(\begin{matrix}f_0,\dots,f_{k} \text{ of degree $d$ } \\ \text{ vanish along $Z$}\end{matrix}\right) + o(\epsilon_{e,k}).
\end{align*}

\paragraph{{\em Medium degree argument}}  We know that $\Prob(\ff\in\cM_{d,k,e})$ is bounded by the sum of the probabilities that $f$ vanishes along some irreducible variety $V$ in $\mathbf V_{e(k+1),k} \setminus \mathbf V_{e,k}$.
 \begin{align*}
\Prob(\ff\in \cM_{d,k,e}) \leq \sum_{\begin{smallmatrix}Z\in \mathbf V_{e(k+1),k} \setminus \mathbf V_{e,k}\end{smallmatrix}}\Prob\left(\begin{matrix}(f_0,f_1,\dots,f_{k}) \text{ of degree $d$ } \\ \text{ vanish along $Z$}\end{matrix}\right).
\end{align*}
Lemma \ref{local_probability} implies that each summand on the right-hand side lies in $o(\epsilon_{e,k})$. This sum is finite and thus $\Prob(\ff\in \cM_{d,k,e})$ is in $o(\epsilon_k)$.

\paragraph{{\em High degree argument}}
Proposition~\ref{prop:precise estimate} implies that $f_0,f_1,\dots,f_{k-1}$ are parameters on $X$ with probability 
$
1-o\left(q^{-\binom{n-k+1+d}{d}}\right)\geq 1-o(\epsilon_{e,k})
$
for any $e$. Hence we may restrict our attention to the case where $f_0,\dots,f_{k-1}$ are parameters on $X$.  

Let $V_1, V_2, \dots, V_s$ be the irreducible components of $
X':=X\cap \VV(f_0,f_1,\ldots,f_{k-1})$ that have dimension $n-k$.
We have that $f_0, f_1\ldots,f_k$ fail to be parameters on $X$ if and only if $f_k$ vanishes on some $V_i$. We can assume that $f_k$ does not vanish on any $V_i$ where $\deg V_i\leq e(k+1)$ as we have already accounted for this possibility in the low and medium degree cases.  After possibly relabelling the components, we let $V_1,V_2,\dots,V_t$ be the components of degree $>e(k+1)$ and $X''=V_1\cup V_2 \cup \cdots \cup V_t$.
Using Lemma~\ref{lem:bezout}, we compute
$
\deghat(X'')\leq \deghat(X') = \deghat(X) \cdot d^k.
$
It follows that $X''$ has at most $\frac{\widehat{\deg(X)}d^k}{e(k+1)}$ irreducible components.
 
Since the value of $d$ is not necessarily larger than the Castelnuovo-Mumford regularity of $V_i$, we cannot use a Hilbert polynomial computation to bound the probability that $f_k$ vanishes along $V_i$. Instead, we use the lower bound for Hilbert functions obtained in Lemma~\ref{lem:key Hilbert function}.  Let $\epsilon=1/2$, though any choice of $\epsilon$ would work. We write $R(V_i)$ for homogeneous coordinate ring of $V_i$.  For any $1\leq i \leq t$, Lemmas~\ref{lem:key Hilbert function} and~\ref{local_probability} yield
\[
\Prob\left(\begin{matrix}f_{k} \text{ of degree $d$ } \\ \text{ vanishes along $V_i$}\end{matrix}\right)=q^{-\dim R(V_i)_d} \leq q^{-(e(k+1)+\epsilon)\binom{n-k+d}{n-k}}
\]
whenever $d\geq Ce^{k+1}$. Combining this with our bound on the number of irreducible components of $X''$ gives
$
\Prob\left(\mathbf{f} \in \cH_{d,k,e}\right) \leq\frac{\deghat X}{e(k+1)}d^k q^{-(e(k+1)+\epsilon)\binom{n-k+d}{n-k}}$ which is in $o(\epsilon_{e,k})$.
\end{proof}

\begin{proof}[Proof of Corollary~\ref{cor:error}]
Let $N$ denote the number of $(n-k)$-planes spaces $L\subseteq \PP^r_{\FF_q}$ such that $L\subseteq X$. Choosing $e=1$ in Theorem~\ref{thm:main finite field}, we compute that
\begin{align*}
\Prob\left(\begin{matrix}f_0,f_1,\dots,f_{k} \text{ of degree $d$ } \\ \text{ are parameters on $X$}\end{matrix}\right) 
&=1 - Nq^{-(k+1)\binom{n-k+d}{n-k}} + o\left(q^{-(k+1)\binom{n-k+d}{n-k}}\right).
\intertext{It follows that}
\Prob\left(\begin{matrix}f_0,f_1,\dots,f_{k} \text{ of degree $d$ } \\ \text{ are \underline{not} parameters on $X$}\end{matrix}\right) 
&= Nq^{-(k+1)\binom{n-k+d}{n-k}} + o\left(q^{-(k+1)\binom{n-k+d}{n-k}}\right).
\end{align*}
Dividing both sides by $q^{-(k+1)\binom{n-k+d}{n-k}}$ and taking the limit as $d\to \infty$ yields the corollary.
\end{proof}

\section{Proofs over $\ZZ$ and $\FF_q[t]$}\label{sec:ZZ}
In this section we prove Theorem~D and Corollary~\ref{thm:noether normalization over ZZ}.

\begin{defn}\label{defn:density}
Let $B=\ZZ$ or $\FF_q[t]$ and fix a finitely generated, free $B$-module $B^s$ and a subset $\mathcal S\subseteq B^s$.  Given $a\in B^s$ we write $a=(a_1,a_2,\dots,a_s)$. The \defi{density} of $\mathcal S\subseteq B$ is
\[
\operatorname{Density}(\mathcal S):=\begin{dcases}
 \lim_{N\to \infty} \frac{\# \{a\in \mathcal S  | \max\{ |a_i| \} \leq N\} }{\# \{a\in \ZZ^s  | \max\{ |a_i| \} \leq N\} } & \text{ if $B=\ZZ$}\\
 \lim_{N\to \infty} \frac{\# \{a\in \mathcal S  | \max\{ \deg a_i  \} \leq N\} }{\# \{a\in \FF_q[t]^s  | \max\{ \deg a_i \} \leq N\} }& \text{ if $B=\FF_q[t]$}\
 \end{dcases}
\]
\end{defn}

\begin{proof}[Proof of Theorem~D]
For clarity, we will prove the result over $\ZZ$ in detail and at the end, mention the necessary adaptions for $\FF_q[t]$.

We first let $k<n$.  Given degree $d$ polynomials $(f_0,f_1,\dots,f_k)$ with integer coefficients and a prime $p$, let $(\overline{f_0},\overline{f_1},\dots,\overline{f_k})$ be the reduction of these polynomials mod $p$. Then $(\overline{f_0},\overline{f_1},\dots,\overline{f_k})$ will be parameters on $X_p$ if and only if the point $\overline{\ff}=(\overline{f}_0,\overline{f}_1,\dots,\overline{f}_k)$ lies $\Par_{d,k}(X_{\FF_p})$. As noted in Remark~\ref{rmk:Par ZZ}, this is equivalent to asking that $\overline{\ff}$ is an $\FF_p$-point of $\Par_{k,d}(X_{\ZZ})$. Thus, we may apply \cite[Theorem~1.2]{ekedahl} to $\Par_{d,k}(X_{\ZZ})\subseteq \mathscr A_{k,d}$ (using $M=1$) to conclude that
\[
 \Density\left\{\begin{matrix}(f_0,\dots,f_{k}) \text{ of degree $d$} \text{ that restrict} \\ \text{ to parameters on $X_p$ for all $p$}\end{matrix}\right\} = \prod_p \Prob \left(\begin{matrix}(f_0,\dots,f_{k}) \text{ of degree $d$} \\ \text{ restrict to  parameters on $X_p$}\end{matrix}\right).
\]
Applying Proposition~\ref{prop:precise estimate} to estimate the individual factors; we have:
\begin{align*}
 \Density\left\{\begin{matrix}(f_0,\dots,f_{k}) \text{ of degree $d$} \text{ that restrict} \\ \text{ to parameters on $X_p$ for all $p$}\end{matrix}\right\} &= \lim_{d\to \infty}\prod_p \Prob \left(\begin{matrix}(f_0,\dots,f_{k}) \text{ of degree $d$} \\ \text{ restrict to  parameters on $X_p$}\end{matrix}\right)\\
 &\geq   \lim_{d\to \infty}\prod_p  \left(1 - \deghat(X_p)(1+d+\dots+d^k)p^{-\binom{n-k+d}{n-k}}\right).\\
\intertext{Lemma~\ref{lem:globalD} shows that there is an integer $D$ where $D\geq \deghat(X_p)$ for all $p$. Moreover, $1+d+\dots+d^k\leq kd^k$ for all $d$, and hence:}
&\geq \lim_{d\to \infty}  \prod_p  \left(1 - Dkd^kp^{-\binom{n-k+d}{n-k}}\right).\\
\intertext{For $d\gg 0$ we can make $Dkd^kp^{-\binom{n-k+d}{n-k}}\leq p^{-d/2}$ for all $p$ simultaneously. Using $\zeta(n)$ for the Riemann-Zeta function, we get:}
&\geq \lim_{d\to \infty}  \prod_p  \left(1 - p^{-d/2}\right)
\geq \lim_{d\to \infty} \zeta(d/2)^{-1}
=1.
 \end{align*}

We now consider the case $k=n$.  This follows by a ``low degree argument'' exactly analogous to~\cite[Theorem~5.13]{poonen}. Fix a large integer $N$ and let $Y$ be the union of all closed points $P\in X$ whose residue field $\kappa(P)$ has cardinality at most $N$.  Since $Y$ is a finite union of closed, we see that for $d\gg 0$, there is a surjection:
\[
\begin{tikzcd}[column sep=3em]
H^0(\PP^r,\cO_{\PP^r}(d)) \rar& H^0(Y,\cO_Y(d))\cong \displaystyle\bigoplus_{\substack{P\in X \\ \#\kappa(P)\leq N}} H^0(P,\cO_P(d))\rar& 0.
\end{tikzcd}
\]
It follows that we have a product formula:
\begin{align*}
\Density\left\{\begin{matrix}(f_0,f_1,\dots,f_{n}) \text{ of degree $d$} \text{ do not vanish} \\ \text{ on a point $P$ with $\#\kappa(P)\leq N$}\end{matrix}\right\} = \prod_{P\in X, \#\kappa(P)\leq N}\left(1-\frac{1}{\#\kappa(P)^{n+1}}\right)
\end{align*}
This is certainly an upper bound on the density of $(f_0,f_1,\dots,f_n)$ that are parameters on $X_p$ for all $p$. As $N\to \infty$ the righthand side approaches $\zeta_{X}(n+1)^{-1}$. However, since the dimension of $X$ is $n+1$, this zeta function has a pole at $s=n+1$~\cite[Theorems~1 and 3(a)]{serre}. Hence this asymptotic density equals $0$.   This completes the proof over $\ZZ$.

Over $\FF_q[t]$, the key adaptation is to use~\cite[Theorem 3.1]{poonen-squarefree} in place of Ekedahl's result.  Poonen's result is stated for a pair of polynomials, but it applies equally well to $n$-tuples of polynomials such as the $n$-tuples defining $\Par_{k,d}(X)$.  In particular, one immediately reduces to proving an analogue of~\cite[Lemma 5.1]{poonen-squarefree}, for $n$-tuples of polynomials which are irreducible over $\FF_q(t)$ and which have gcd equal to $1$; but the $n=2$ version of the lemma then implies the $n\geq 2$ versions of the lemma.\footnote{We thank Bjorn Poonen for pointing out this reduction.}  The rest of our argument over $\ZZ$ works over $\FF_q[t]$.
\end{proof}

\begin{lemma}\label{lem:globalD}
Let $X\subseteq \PP^r_{B}$ be any closed subscheme. There is an integer $D$ where $D\geq \deghat(X_s)$ for all $s\in \Spec B$.
\end{lemma}
\begin{proof}
First we take a flattening stratification for $X$ over $B$~\cite[Corollaire~6.9.3]{ega44}. Within each strata, the maximal degree of a minimal generator is semicontinuous, and we can thus find a degree $e$ where $X_s$ is generated in degree $e$ for all $s\in \Spec B$.  
By~\cite[Prop.~3.5]{bayer-mumford}, we then obtain that $\deghat(X) \leq \sum_{j=0}^n e^{r-j}$. In particular defining $D:=re^r$ will suffice.
\end{proof}

To prove Corollary~\ref{thm:noether normalization over ZZ}, we use Theorem~D to find a submaximal collection $(f_0,f_1,\dots,f_{n-1})$ which restrict to parameters on $X_s$ for all $s\in \Spec B$. This cuts $X$ down to a scheme $X'=X\cap \VV(f_0,f_1,\dots,f_{n-1})$ with $0$-dimensional fibers over each point $s$.  When $B=\ZZ$, such a scheme is essentially a union of orders in number fields, and we find the last element $f_n$ by applying classical arithmetic results about the Picard groups of rings of integers of number fields.  When $B=\FF_q[t]$, we use similar facts about Picard groups of affine curves over $\FF_q$.

An example illustrates this approach. Let $X=\PP^1_{\ZZ}=\Proj(\ZZ[x,y])$. A polynomial of degree $d$ will be a parameter on $X$ as long as the $d+1$ coefficients are relatively prime. Thus as $d\to \infty$, the density of these choices will go to $1$.  However, once we have fixed one such parameter, say $5x-3y$, it is much harder to find an element that will restrict to a parameter on $\ZZ[x,y]/(5x-3y)$ modulo $p$ for all $p$.  In fact, the only possible choices are the elements which restrict to units on $\Proj(\ZZ[x,y]/(5x-3y))$. Among the linear forms, these are
\[
\pm (7x-4y) + c(5x-3y)  \text{ for any } c\in \ZZ.
\]
Hence, these elements arise with density zero, and yet they form a nonempty subset.

Lemmas~\ref{lem:Pic torsion} and \ref{lem:Pic torsion curves} below are well-known to experts, but we sketch the proofs for clarity.
\begin{lemma}\label{lem:Pic torsion}
If $X'\subseteq \PP^r_{\ZZ}$ is closed and finite over $\Spec(\ZZ)$, then $\Pic(X')$ is finite.
\end{lemma}
\begin{proof}
We first reduce to the case where $X'$ is reduced. Let $\mathcal N\subseteq \cO_{X'}$ be the nilradical ideal. If $X$ is nonreduced then there is some integer $m>1$ for which $\mathcal N^m=0$. Let $X''$ be the closed subscheme defined by $\mathcal N^{m-1}$. 
We have a short exact sequence
$
0\to \mathcal N^{m-1}\to \cO_{X'}^*\to \cO_{X''}^*\to 1
$
where the first map sends $f\mapsto 1+f$.
Since $H^1(X',\mathcal N^{m-1})=H^2(X',\mathcal N^{m-1})=0$, taking cohomology yields an isomorphism $\Pic(X')\cong \Pic(X'')$. Iterating this argument, we may assume $X'$ is reduced.

We now have $X'=\Spec(B)$ where $B$ is a finite, reduced $\ZZ$-algebra. If $Q$ is a minimal prime of $B$, then $B/Q$ is either zero dimensional or an order in a number field, and hence has a finite Picard group~\cite[Theorem~12.12]{neukirch}.
If $Q'$ is the intersection of all of the other minimal primes $B$. Then we again have an exact sequence in cohomology
%
\[
\begin{tikzcd}[column sep=1.5em]
\dots \rar& (B/(Q+Q'))^* \rar& \Pic(X') \rar& \Pic(B/Q)\oplus \Pic(B/Q') \rar& \dots
\end{tikzcd}
\]
Since $(B/(Q+Q'))^*$ is a finite set, and since $B/Q$ and $B/Q'$ have fewer minimal primes than $B$, we may use induction to conclude that $\Pic(X')$ is finite.
\end{proof}

\begin{lemma}\label{lem:Pic torsion curves}
If $C$ is an affine curve over $\FF_q$, then $\Pic(C)$ is finite.
\end{lemma}
\begin{proof}
If $C$ fails to be integral, then an argument entirely analogous to the proof of Lemma~\ref{lem:Pic torsion} reduces us to the case $C$ is integral.  
We next assume that $C$ is nonsingular and integral, and that $\overline{C}$ is the corresponding nonsingular projective curve.  Since $C$ is affine we have
$\Pic(C)=\Pic^0(C)\subseteq \Pic^0(\overline{C})\subseteq \operatorname{Jac}(\overline{C})(\FF_q),
$ which is a finite set.
If $C$ is singular, then the finiteness of $\Pic(C)$ follows from the nonsingular case by a minor adapation of the proof of~\cite[Proposition 12.9]{neukirch}.
\end{proof}

\begin{proof}[Proof of Corollary~\ref{thm:noether normalization over ZZ}]
By Theorem~D, for $d\gg 0$ we can find polynomials $f_0,f_1,\dots,f_{n-1}$ of degree $d$ that restrict to parameters on $X_s$ for all $s\in \Spec B$. Let $X':=\VV(f_0,f_1,\dots,f_{n-1})\cap X$, which is finite over $B$ by construction. Let $A$ be the finite $B$-algebra where $\Spec A =X'$.  
Lemma~\ref{lem:Pic torsion} or ~\ref{lem:Pic torsion curves}  implies that $H^0(X',\cO_{X'}(e))=A$ for some $e$.  We can thus find a polynomial $f_n$ of degree $e$ mapping onto a unit in the $B$-algebra $A$.  It follows that $\VV(f_n)\cap X'=\emptyset$.  Replace $f_i$ by $f_i^{e}$ for $i=0,\dots,n-1$ and replace $f_n$ by $f_n^{d}$. Then we have $f_0,f_1,\dots,f_n$ of degree $d':=de$ and restricting to parameters on $X_s$ for all $s\in \Spec(B)$ simultaneously.

We thus obtain a proper morphism $\pi: X\to \PP^n_{B}$ where $X_s\to \PP^n_{\kappa(s)}$ is finite for all $s$.  Since $\pi$ is quasi-finite and proper, it is finite by \cite[Th\'{e}or\`{e}me 8.11.1]{ega43}.
\end{proof}

The following generalizes Corollary~\ref{thm:noether normalization over ZZ} to other graded rings.

\begin{cor}\label{cor:graded rings}
Let $B=\ZZ$ or $\FF_q[t]$ and let $R$ be a graded, finite type $B$-algebra where $\dim R\otimes_{\ZZ} \FF_p=n+1$ for all $p$.  Then there exist $f_0,f_1,\dots,f_n$ of degree $d$ for some $d$ such that $B[f_0,f_1,\dots,f_n]\subseteq R$ is a finite extension.
\end{cor}
\begin{proof}
After replacing $R$ by a high degree Veronese subring $R'$, we may assume that $R'$ is generated in degree one and contains no $R'_{+}$-torsion submodule, where $R'_+\subseteq R'$ is the homogeneous ideal of strictly positive degree elements.  Let $r+1$ be the number of generators of $R'_1$.  Then there is a surjection $\phi\colon B[x_0,\dots,x_r]\to R'$ inducing an embedding of $X:=\Proj(R')\subseteq \PP^r_{B}$.  Since $R'$ contains no $R'_{+}$-torsion submodule, the kernel of $\phi$ will be saturated with respect to $(x_0,x_1,\dots,x_r)$ and hence $R'$ will equal the homogeneous coordinate ring of $X$.  Choosing $f_0,f_1,\dots,f_n$ as in Corollary~\ref{thm:noether normalization over ZZ}, it follows that $B[f_0,f_1,\dots,f_n]\subseteq R'$ is a finite extension, and thus so is $B[f_0,f_1,\dots,f_n]\subseteq R$.
\end{proof}

\section{Examples}\label{sec:examples}

\begin{example}\label{ex:linear}
By Corollary~\ref{cor:error}, it is more difficult to randomly find parameters on surfaces that contain lots of lines.  Consider $\VV(xyz)\subset \PP^3$ which contains substantially more lines than $\VV(x^2+y^2+z^2)\subset\PP^3$.  Using Macaulay2~\cite{M2} to select 1,000,000 random pairs $(f_0,f_1)$ of polynomials of degree two, the proportion that failed to be systems of parameters were:
\[
\begin{tabular}{| c | c | c | }\hline
& $\VV(xyz)$ & $\VV(x^2+y^2+z^2)$ \\ \hline
$\FF_2$ & .2638 & .1179 \\ \hline
$\FF_3$ & .0552 & .0059\\ \hline
$\FF_5$ & .0063 & .0004 \\ \hline
\end{tabular}
\]
\end{example}

\begin{example}
Let $X\subseteq \PP^3_{\FF_q}$ be a smooth cubic surface. Over the algebraic closure $X$ has 27 lines, but it has between $0$ and $27$ lines defined over $\FF_q$. For example, working over $\FF_4$, the Fermat cubic surface $X'$ defined by $x^3+y^3+z^3+w^3$ has 27 lines, while the cubic surface $X$ defined by $x^3+y^3+z^3+aw^3$ where $a\in \FF_4\setminus\FF_2$ has no lines defined over $\FF_4$ \cite{debarre-laface}. It will thus be more difficult to find parameters on $X$ than on $X'$. 
Using Macaulay2~\cite{M2} to select 100,000 random pairs $(f_0,f_1)$ of polynomials of degree two, 0.62\% failed to be parameters on $X$ whereas no choices whatsoever failed to be parameters on $X'$.  This is in line with the predictions from Corollary~\ref{cor:error}; for instance, in the case of $X$, we have $27\cdot 4^{-2\cdot 3}\approx 0.66\%$.
\end{example}

\begin{example}\label{ex:deg60}
Let $X=[1:4]\cup [3:5]\cup [4:5] = \mathbb V((4x-y)(5x-3y)(5x-4y))\subseteq \PP^1_{\ZZ}$ and let $R$ be the homogeneous coordinate ring of $X$.  The fibers are $0$-dimensional so finding a Noether normalization $X\to \PP^0_{\ZZ}$ is equivalent to finding a single polynomial $f_0$ that restricts to a unit on all of the points simultaneously.  We can find such an $f_0$ of degree $d$ if and only if the induced map of free $\ZZ$-modules $\ZZ[x,y]_d \to R_d$ is surjective.  A computation in Macaulay2~\cite{M2} shows that this happens if and only if $d$ is divisible by $60$.
\end{example}

\begin{example}\label{ex:flat ZZ}
Let $R=\ZZ[x]/(3x^2-5x)\cong \ZZ\oplus \ZZ[\frac{1}{3}]$. This is a flat, finite type $\ZZ$-algebra where every fiber has dimension $0$, yet it is not a finite extension of $\ZZ$.  However, if we take the projective closure of $\Spec(R)$ in $\PP^1_{\ZZ}$, then we get $\Proj(\overline{R})$ where $\overline{R}=\ZZ[x,y]/(3x^2-5xy)$. If we then choose $f_0:=4x-7y$, we see that $\ZZ[f_0]\subseteq \overline{R}$ is a finite extension of graded rings.
\end{example}

\begin{example}
Let $\kk$ be a field and let $X=[1:1+t]\cup [1-t:1]=\mathbb V((y-(1+t)x)(x-(1-t)y))\subseteq \PP^1_{\kk[t]}$.  Let $R$ be the homogeneous coordinate ring of $X$.  In degree $d$, we have the map
$
\phi_d: \kk[t][x,y]_d \cong \kk[t]^{d+1}\to R_d\cong \kk[t]^2.
$
Choosing the standard basis $x^d, x^{d-1}y, \dots, y^d$ for the source of $\phi_d$, and the two points of $X$ for the target, we can represent $\phi_d$ by the matrix
\[
\begin{pmatrix}
1&1+t&(1+t)^2&\dots & (1+t)^d\\
(1-t)^d & (1-t)^{d-1}&(1-t)^{d-2}&\dots & 1
\end{pmatrix}.
\]
It follows that
$
\im \phi_d =\im \begin{pmatrix}t^2 & (1+t)^d\\0&1  \end{pmatrix} = \im \begin{pmatrix}t^2 & 1+dt\\0&1  \end{pmatrix}.
$
The image of $\phi_d$ thus contains a unit if and only if the characteristic of $\kk$ is $p$ and $p|d$.  In particular, if $\kk=\QQ$, then we cannot find a polynomial $f_0$ inducing a finite map $X\to \PP^0_{\QQ[t]}$.
\end{example}

\begin{example}\label{example:failure-high-dim}
Let $\kk$ be any field, let $B=\kk[s,t]$, and let $X=[s:1]\cup [1:t] =\mathbb V((x-sy)(y-tx))\subseteq \PP^1_{B}$.  We claim that for any $d>0$, there does not exist a polynomial that restricts to a parameter on $X_b$ for each point $b\in B$.   Assume for contradiction that we had such an $f=\sum_{i=0}^d c_is^it^{d-i}$ with $c_i\in B$.  After scaling, we obtain
\[
f([s:1])=c_0s^d+c_1s^{d-1}+\cdots+c_d=1\quad \text{ and } \quad
f([1:t])=c_0+c_1t+\cdots+c_dt^d=\lambda
\]
where $\lambda\in B^*=\kk^*$.  Substituting for $c_d$ we obtain
\begin{align*}
f([1:t])
&=c_0+c_1t+\cdots+c_{d-1}t^{d-1}+\left(1-\left(c_0s^d+c_1s^{d-1}+\cdots+c_{d-1}s\right)\right)t^d=\lambda,
\end{align*}
which implies that
\begin{align*}
\lambda-t^d&=c_0+c_1t+\cdots+c_{d-1}t^{d-1}-\left(c_0s^d+c_1s^{d-1}+\cdots+c_{d-1}s\right)t^d\\
&=(c_0-c_0s^dt^d)+(c_1t-c_1s^{d-1}t^d)+\cdots+(c_{d-1}t^{d-1}-c_{d-1}st^d)
=(1-st)h(s,t)
\end{align*}
where $h(s,t)\in \kk[s,t]$. This implies that $\lambda-t^d$ is divisible by $(1-st)$, which is a contradiction.
\end{example}


\begin{bibdiv}
\begin{biblist}

\bib{abhyankar-kravitz}{article}{
   author={Abhyankar, Shreeram S.},
   author={Kravitz, Ben},
   title={Two counterexamples in normalization},
   journal={Proc. Amer. Math. Soc.},
   volume={135},
   date={2007},
   number={11},
   pages={3521--3523},
}


\bib{achinger}{article}{
   author={Achinger, Piotr},
   title={$K(\pi,1)$-neighborhoods and comparison theorems},
   journal={Compos. Math.},
   volume={151},
   date={2015},
   number={10},
   pages={1945--1964},
}
\bib{bayer-mumford}{article}{
   author={Bayer, Dave},
   author={Mumford, David},
   title={What can be computed in algebraic geometry?},
   conference={
      title={Computational algebraic geometry and commutative algebra
      (Cortona, 1991)},
   },
   book={
      series={Sympos. Math., XXXIV},
      publisher={Cambridge Univ. Press, Cambridge},
   },
   date={1993},
   pages={1--48},
}

\bib{benoist}{article}{
   author={Benoist, Olivier},
   title={Le thŽ\'eorem\`e de Bertini en famille},
   language={French, with English and French summaries},
   journal={Bull. Soc. Math. France},
   volume={139},
   date={2011},
   number={no.~4},
   pages={555--569},
}

\bib{brennan-epstein}{article}{
   author={Brennan, Joseph P.},
   author={Epstein, Neil},
   title={Noether normalizations, reductions of ideals, and matroids},
   journal={Proc. Amer. Math. Soc.},
   volume={139},
   date={2011},
   number={8},
   pages={2671--2680},
}

\bib{bruns-herzog}{book}{
   author={Bruns, Winfried},
   author={Herzog, J{\"u}rgen},
   title={Cohen-Macaulay rings},
   series={Cambridge Studies in Advanced Mathematics},
   volume={39},
   publisher={Cambridge University Press, Cambridge},
   date={1993},
   pages={xii+403},
}

\bib{bucur-kedlaya}{article}{
   author={Bucur, Alina},
   author={Kedlaya, Kiran S.},
   title={The probability that a complete intersection is smooth},
   language={English, with English and French summaries},
   journal={J. Th\'eor. Nombres Bordeaux},
   volume={24},
   date={2012},
   number={3},
   pages={541--556},
}

\bib{cmbpt}{article}{
	author = {Chinburg, Ted},
	author = {Moret-Bailly, Laurent},
	author = {Pappas, George},
	author = {Taylor, Martin J.~},
	title = {Finite morphisms to projective space and capacity theory},
	date = {2012},
	note = {arXiv:1201.0678},
}

\bib{debarre-laface}{article}{
   author={Debarre, Alina},
   author={Laface, Antonio},
   author={Roulleau, Xavier},
   title={Lines on cubic hypersurfaces over finite fields},
   language={English, with English and French summaries},
   journal={eprint arXiv:1510.05803},
   date={2015},
}


\bib{eisenbud}{book}{
   author={Eisenbud, David},
   title={Commutative algebra},
   series={Graduate Texts in Mathematics},
   volume={150},
   note={With a view toward algebraic geometry},
   publisher={Springer-Verlag, New York},
   date={1995},
   pages={xvi+785},
}


\bib{ekedahl}{article}{
   author={Ekedahl, Torsten},
   title={An infinite version of the Chinese remainder theorem},
   journal={Comment. Math. Univ. St. Paul.},
   volume={40},
   date={1991},
   number={1},
   pages={53--59},
}

\bib{ellenberg-erman}{article}{
   author={Ellenberg, Jordan S.},
   author={Erman, Daniel},
   title={Furstenberg sets and Furstenberg schemes over finite fields},
   journal={Algebra Number Theory},
   volume={10},
   date={2016},
   number={7},
   pages={1415--1436},
}


\bib{eot}{article}{
   author={Ellenberg, Jordan S.},
   author={Oberlin, Richard},
   author={Tao, Terence},
   title={The Kakeya set and maximal conjectures for algebraic varieties
   over finite fields},
   journal={Mathematika},
   volume={56},
   date={2010},
   number={1},
   pages={1--25},
}


%


\bib{fulton}{book}{
   author={Fulton, William},
   title={Intersection theory},
   series={Ergebnisse der Mathematik und ihrer Grenzgebiete. 3. Folge. A
   Series of Modern Surveys in Mathematics [Results in Mathematics and
   Related Areas. 3rd Series. A Series of Modern Surveys in Mathematics]},
   volume={2},
   edition={2},
   publisher={Springer-Verlag, Berlin},
   date={1998},
   pages={xiv+470},
}

\bib{gkz}{book}{
   author={Gelfand, I. M.},
   author={Kapranov, M. M.},
   author={Zelevinsky, A. V.},
   title={Discriminants, resultants and multidimensional determinants},
   series={Modern Birkh\"auser Classics},
   note={Reprint of the 1994 edition},
   publisher={Birkh\"auser Boston, Inc., Boston, MA},
   date={2008},
   pages={x+523},
}


\bib{gabber-liu-lorenzini}{article}{
   author={Gabber, Ofer},
   author={Liu, Qing},
   author={Lorenzini, Dino},
   title={Hypersurfaces in projective schemes and a moving lemma},
   journal={Duke Math. J.},
   volume={164},
   date={2015},
   number={7},
   pages={1187--1270},
}
	

\bib{ega43}{article}{
   author={Grothendieck, A.},
   author={Dieudonn\'e, J.},
   title={\'El\'ements de g\'eom\'etrie alg\'ebrique. IV. \'Etude locale des
   sch\'emas et des morphismes de sch\'emas. III},
   journal={Inst. Hautes \'Etudes Sci. Publ. Math.},
   number={28},
   date={1966},
   pages={255},
}

\bib{ega44}{article}{
   author={Grothendieck, A.},
      author={Dieudonn\'e, J.},
   title={\'El\'ements de g\'eom\'etrie alg\'ebrique. IV. \'Etude locale des
   sch\'emas et des morphismes de sch\'emas IV},
   language={French},
   journal={Inst. Hautes \'Etudes Sci. Publ. Math.},
   number={32},
   date={1967},
   pages={361},
}

\bib{hartshorne}{book}{
   author={Hartshorne, Robin},
   title={Algebraic Geometry},
   series={Graduate Texts in Mathematics},
   volume={52},
   note={With a view toward algebraic geometry},
   publisher={Springer-Verlag, New York},
   date={1977},
   pages={xvi+496},
}

\bib{kedlaya-more-etale}{article}{
   author={Kedlaya, Kiran S.},
   title={More \'etale covers of affine spaces in positive characteristic},
   journal={J. Algebraic Geom.},
   volume={14},
   date={2005},
   number={1},
   pages={187--192},
}

\bib{lang-weil}{article}{
   author={Lang, Serge},
   author={Weil, Andr{\'e}},
   title={Number of points of varieties in finite fields},
   journal={Amer. J. Math.},
   volume={76},
   date={1954},
   pages={819--827},
}

\bib{moh}{article}{
   author={Moh, T. T.},
   title={On a normalization lemma for integers and an application of four
   colors theorem},
   journal={Houston J. Math.},
   volume={5},
   date={1979},
   number={1},
   pages={119--123},
}


\bib{macaulay-determinantal}{book}{
   author={Macaulay, F. S.},
   title={The algebraic theory of modular systems},
   series={Cambridge Mathematical Library},
   note={Revised reprint of the 1916 original;
   With an introduction by Paul Roberts},
   publisher={Cambridge University Press, Cambridge},
   date={1994},
   pages={xxxii+112},
}

\bib{M2}{misc}{
    label={M2},
    author={Grayson, Daniel~R.},
    author={Stillman, Michael~E.},
    title = {Macaulay 2, a software system for research
	    in algebraic geometry},
    note = {Available at \url{http://www.math.uiuc.edu/Macaulay2/}},
}

\bib{nagata}{book}{
   author={Nagata, Masayoshi},
   title={Local rings},
   series={Interscience Tracts in Pure and Applied Mathematics, No. 13},
   publisher={Interscience Publishers a division of John Wiley \& Sons\, New
   York-London},
   date={1962},
}

\bib{neukirch}{book}{
   author={Neukirch, J{\"u}rgen},
   title={Algebraic number theory},
   series={Grundlehren der Mathematischen Wissenschaften [Fundamental
   Principles of Mathematical Sciences]},
   volume={322},
   note={Translated from the 1992 German original and with a note by Norbert
   Schappacher;
   With a foreword by G. Harder},
   publisher={Springer-Verlag, Berlin},
   date={1999},
   pages={xviii+571},
}
%

\bib{poonen}{article}{
   author={Poonen, Bjorn},
   title={Bertini theorems over finite fields},
   journal={Ann. of Math. (2)},
   volume={160},
   date={2004},
   number={3},
   pages={1099--1127},
}

\bib{poonen-squarefree}{article}{
   author={Poonen, Bjorn},
   title={Squarefree values of multivariable polynomials},
   journal={Duke Math. J.},
   volume={118},
   date={2003},
   number={2},
   pages={353--373},
}

\bib{serre}{article}{
   author={Serre, Jean-Pierre},
   title={Zeta and $L$ functions},
   conference={
      title={Arithmetical Algebraic Geometry},
      address={Proc. Conf. Purdue Univ.},
      date={1963},
   },
   book={
      publisher={Harper \& Row, New York},
   },
   date={1965},
   pages={82--92},
}


\bib{zariski-samuel}{book}{
   author={Zariski, Oscar},
   author={Samuel, Pierre},
   title={Commutative algebra. Vol. II},
   note={Reprint of the 1960 edition;
   Graduate Texts in Mathematics, Vol. 29},
   publisher={Springer-Verlag, New York-Heidelberg},
   date={1975},
   pages={x+414},
}
\end{biblist}
\end{bibdiv}
\end{document}